\documentclass{article}
\usepackage{amsmath,amssymb,amsthm}
\usepackage{appendix}

\usepackage{xcolor}

\newtheorem{theorem}{Theorem}
\newtheorem{proposition}{Proposition}
\newtheorem{corollary}{Corollary}
\newtheorem{definition}{Definition}
\theoremstyle{remark}
\newtheorem{remark}{Remark}
\newtheorem{example}{Example}

\def\circ{\mathrm{circ}}

\title{Circulant matrices with orthogonal rows\\and off-diagonal entries of absolute value $1$}
\author{Daniel Uzc\'{a}tegui Contreras, Dardo Goyeneche, Ond\v{r}ej Turek, Zuzana V\'{a}clav\'{i}kov\'{a}}

\begin{document}

\maketitle

\begin{abstract}
It is known that a real symmetric circulant matrix with diagonal entries $d\geq0$, off-diagonal entries $\pm1$ and orthogonal rows exists only of order $2d+2$ (and trivially of order $1$) [Turek and Goyeneche 2019].
In this paper we consider a complex Hermitian analogy of those matrices. That is, we study the existence and construction of Hermitian circulant matrices having orthogonal rows, diagonal entries $d\geq0$ and any complex entries of absolute value $1$ off the diagonal. As a particular case, we consider matrices whose off-diagonal entries are 4th roots of unity; we prove that the order of any such matrix with $d$ different from an odd integer is $n=2d+2$. We also discuss a similar problem for symmetric circulant matrices defined over finite rings $\mathbb{Z}_m$. As an application of our results, we show a close connection to mutually unbiased bases, an important open problem in quantum information theory.
\end{abstract}

\section{Introduction}

A \emph{circulant matrix} is a square matrix of order $n\in\mathbb{N}$ of the form
\begin{equation}\label{C}
C=\left(\begin{array}{ccccc}
c_{0} & c_{1} & \cdots & c_{n-2} & c_{n-1} \\
c_{n-1} & c_{0} & c_{1} &  & c_{n-2} \\
\vdots & c_{n-1} & c_{0} & \ddots & \vdots \\
c_{2} &  & \ddots & \ddots & c_{1} \\
c_{1} & c_{2} & \cdots & c_{n-1} & c_{0}
\end{array}\right)\,.
\end{equation}
The first row, $(c_0,c_1,\ldots,c_{n-1})$, is called the \emph{generator} of $C$. In this work we will denote a circulant matrix of order $n$ having generator $(c_0,c_1,\ldots,c_{n-1})$ by $\circ_n(c_0,c_1,\ldots,c_{n-1})$.

Let $C= \circ_n(c_0,c_1,\ldots,c_{n-1})$ be a complex circulant matrix of order $n\geq2$ satisfying the following conditions:
\begin{equation}\label{Conditions}
\left\{
\begin{array}{l}
c_0=d\geq 0; \\
|c_j|=1 \quad\text{for all}\; j=1,\ldots,n-1; \\
CC^*=(d^2+n-1)I.
\end{array}
\right.
\end{equation}
The aim of this paper is to examine possible orders of a matrix $C$ obeying the above conditions for a given $d$. In other words, we shall examine values of $d$ that are allowed on the main diagonal of matrices $C$ of a given order $n$.\medskip

In paper \cite{TG}, real matrices satisfying \eqref{Conditions} were studied. In particular, a complete solution was obtained for the case of \emph{symmetric} matrices. It was proved that the order of a symmetric matrix $C$ is related with the diagonal value $d$ by the formula $n=2d+2$. In the present work we discuss extensions of the results in two directions:
\begin{itemize}
\item complex Hermitian matrices;
\item symmetric matrices with entries defined over finite rings $\mathbb{Z}_m$.
\end{itemize}

Note that a complex \emph{non-Hermitian} matrix $C$ satisfying \eqref{Conditions} with a given $d\geq0$ trivially exists for every $n\leq2d+2$. Indeed, consider $C = \circ_n(d,-\mathrm{e}^{\mathrm{i}\alpha},-\mathrm{e}^{\mathrm{i}\alpha},\ldots,-\mathrm{e}^{\mathrm{i}\alpha})$. Then $CC^*$ is a circulant matrix with generator
$$
(d^2+n-1,n-2-2d\cos\alpha,n-2-2d\cos\alpha,\ldots,n-2-2d\cos\alpha);
$$
so every $n\leq 2d+2$ allows to set $\alpha=\arccos\frac{n-2}{2d}$ to obtain a matrix $C$ satisfying \eqref{Conditions}. On the other hand, the question becomes hard for orders $n>2d+2$.

Our object of study has a close relation with \textit{polyphase sequences} \cite{heimiler1961,mow1993}, that is, \textit{n}-tuple sequences of complex numbers having the form $\omega^k$, where $\omega=\exp{(\frac{2\pi\mathrm{i}}{n})}$ is the main $n$th root of the unity. Among the entire set of polyphase sequences there is a relevant subset given by \textit{perfect autocorrelation sequences}, which are characterized by having zero autocorrelation function \cite{heimiler1961}. Let us recall that for a sequence $\mathbf{a} = (a_0, a_1,...,a_{n-1})$, whose elements satisfy $a_i = a_{i + \nu}$, the autocorrelation function $\theta_{\mathbf{a}}(\nu)$ is defined as\\ 

\begin{equation}
    \theta_{\mathbf{a}}(\nu) = \sum_{i = 0}^{n - 1} a_i a^{*}_{i + \nu },
\end{equation}

\noindent where $\nu$ is called the shift or period and  $i + \nu$ is computed modulo $n$ \cite{blake2014}. In a sense, the autocorrelation function quantifies how much a sequence differs from its cyclic shifts of entries. Polyphase sequences having perfect autocorrelation are one-to-one connected with generators $g$ of matrices $C$ having order $n$ satisfying conditions (\ref{Conditions}) for the special case of $d=1$ and $n$th roots of the unity in its entries. These sequences have practical applications in several fields, for example in communication and radar systems \cite{farnertt1990,itapov2005,milewski1983,xu2011}. Therefore, construction of perfect sequences of length $n$ has been extensively studied (cf. \cite{chu1972,heimiler1961,LF2004,milewski1983} and references therein).

\section{Preliminaries}

A circulant matrix $C$ of order $n$ has normalized eigenvectors $v_0,v_1,\ldots,v_{n-1}$ given as
$$
v_k=\frac{1}{\sqrt{n}}\left(1,\omega^k,\omega^{2k},\ldots,\omega^{(n-1)k}\right)^T\,,
$$
where $\omega=\mathrm{e}^{\frac{2\pi\mathrm{i}}{n}}$.
The associated eigenvalues are
\begin{equation}\label{eigenvalues}
\lambda_k=c_0+c_{1}\omega^k+c_{2}\omega^{2k}+\cdots+c_{n-1}\omega^{(n-1)k}\,,
\end{equation}
where $(c_0,c_1,\ldots,c_{n-1})$ is the generator of $C$.

The vectors $(c_0,c_1,\ldots,c_{n-1})^T$ and $(\lambda_0,\lambda_1,\ldots,\lambda_{n-1})^T$ are related by the discrete Fourier transform; the inverse transform gives the generator in terms of the eigenvalues as follows:
\begin{equation}\label{invDFT}
c_j=\frac{1}{n}\left(\lambda_0+\lambda_1\omega^{-j}+\lambda_2\omega^{-2j}+\cdots+\lambda_{n-1}\omega^{-(n-1)j}\right).
\end{equation}

Throughout the paper, we will index the rows and columns of the matrix $C$ by integers from $0$ to $n-1$ (instead of from $1$ to $n$).

\section{Hermitian solutions over $\mathbb{C}$}

First of all, let us observe that for each $n\geq2$ there exists a Hermitian circulant matrix satisfying \eqref{Conditions} with main diagonal $d=\frac{n}{2}-1$.

\begin{proposition}\label{Prop. trivial}
A Hermitian circulant matrix
$$
C=\circ_n\left(\frac{n}{2}-1,-\omega^{\nu},-\omega^{2\nu},\ldots,-\omega^{(n-1)\nu}\right)
$$
where $\omega=\mathrm{e}^{2\pi \mathrm{i}/n}$ satisfies conditions \eqref{Conditions} for every $n\geq2$ and every $\nu\in\mathbb{Z}_n$. In particular, the choice $\nu=0$ gives a real symmetric solution
$$
C= \circ_n\left(\frac{n}{2}-1,-1,-1,\ldots,-1\right).
$$
\end{proposition}

\begin{proof}
From equation \eqref{eigenvalues}, eigenvalues of $C$ can be written in term of entries $c_j$ as follows:

\begin{equation}
\lambda_k = \sum_{j = 0}^{n-1}c_j \omega^{jk}= \frac{n}{2}-1 + \sum_{j = 1}^{n-1}(-\omega^{j\nu})\omega^{jk} = \frac{n}{2}-1 - \sum_{j = 0}^{n-1}\omega^{j(\nu+k)} + 1=\frac{n}{2}-n\,\delta_{\nu+k,0}.
\label{eigens}
\end{equation}
This equation becomes either $\lambda_k = \frac{n}{2}$ or  $\lambda_{k} = \frac{n}{2} - n = -\frac{n}{2}$ for $\nu+k \neq 0$ or $\nu+k=0$, respectively.
Thus $CC^*=\left(\frac{n}{2}\right)^2 I = \left(\left(\frac{n}{2}-1\right)^2+n-1\right)^2 I$, so $C$ satisfies conditions \eqref{Conditions} for every $\nu\in\mathbb{Z}_n$.
\end{proof}

Let us now examine the situation of a general $d\geq0$. We will distinguish matrices of even and odd orders.

\subsection{Matrices $C$ of even orders}\label{even orders}

\begin{proposition}\label{Prop. even n}
If a Hermitian circulant matrix $C$ of an even order $n$ satisfies \eqref{Conditions}, then
\begin{itemize}
\item[(i)] there exists a positive integer $k\leq\frac{n}{2\sqrt{n-1}}$ such that
\begin{equation}\label{d2}
\sqrt{d^2+n-1}=\frac{n}{2k};
\end{equation}
\item[(ii)] $d$ is rational;
\item[(iii)] $d\leq\frac{n}{2}-1$.
\end{itemize}
\end{proposition}

\begin{proof}
Since $C$ is Hermitian and satisfies $C^2=(d^2+n-1)I$, the eigenvalues of $C$ are $\sqrt{d^2+n-1}$ and $-\sqrt{d^2+n-1}$. Let us denote their multiplicities by $\nu$ and $n-\nu$, respectively. The sum of eigenvalues is equal to the trace of $C$, i.e.,
\begin{equation}\label{sum=trace C}
\nu\sqrt{d^2+n-1}-(n-\nu)\sqrt{d^2+n-1}=nd.
\end{equation}
Hence
\begin{equation}\label{sum=trace C 2}
\left(\nu-\frac{n}{2}\right)\sqrt{d^2+n-1}=\frac{n}{2}d.
\end{equation}
Now we will use an idea from~\cite[proof of Theorem~8]{Cr}. Since $C$ is Hermitian, its generator has the form $(d,c_1,\ldots,c_{\frac{n}{2}-1},c_{\frac{n}{2}},\overline{c_{\frac{n}{2}-1}},\ldots,\overline{c_1})$, where $c_{\frac{n}{2}}\in\mathbb{R}$. Let $M$ be a circulant matrix with the generator
$$
(c_{\frac{n}{2}},\overline{c_{\frac{n}{2}-1}},\ldots,\overline{c_1},d,c_1,\ldots,c_{\frac{n}{2}-1}).
$$
Note that $M$ is Hermitian and satisfies $M=CP$, where $P=\begin{pmatrix}0&I\\I&0\end{pmatrix}$ is a permutation matrix. Therefore,
$$
M^2=MM^*=(CP)(CP)^*=CPP^*C^*=CC^*=(d^2+n-1)I.
$$
Consequently, $M$ has eigenvalues $\sqrt{d^2+n-1}$ and $-\sqrt{d^2+n-1}$; we denote their multiplicities by $\mu$ and $n-\mu$, respectively. The sum of eigenvalues of $M$ must be equal to the trace of $M$, so
\begin{equation}\label{sum=trace M}
\mu\sqrt{d^2+n-1}-(n-\mu)\sqrt{d^2+n-1}=nc_{\frac{n}{2}}.
\end{equation}
Recall that $c_{\frac{n}{2}}$ is real due to the hermiticity of $M$. At the same time $|c_{\frac{n}{2}}|=1$ by \eqref{Conditions}. Hence $c_{\frac{n}{2}}=\pm1$, and equation \eqref{sum=trace M} implies
\begin{equation}\label{sum=trace M 2}
\left|\mu-\frac{n}{2}\right|\sqrt{d^2+n-1}=\frac{n}{2}.
\end{equation}
We denote $k:=\left|\mu-\frac{n}{2}\right|$. By definition of $\mu$, $k$ is an integer from $\left[0,\frac{n}{2}\right]$. The value $k=0$ is forbidden by \eqref{sum=trace M 2}. Values $k>\frac{n}{2\sqrt{n-1}}$ would imply $\sqrt{d^2+n-1}<\sqrt{n-1}$, which is impossible. So $1\leq k\leq\frac{n}{2\sqrt{n-1}}$.

(ii)\; Since $\sqrt{d^2+n-1}=\frac{n}{2k}\in\mathbb{Q}$, equation \eqref{sum=trace C 2} gives $d\in\mathbb{Q}$.

(iii)\; Applying statement (i), one gets
$$
d=\sqrt{\left(\frac{n}{2k}\right)^2-n+1}\leq\sqrt{\left(\frac{n}{2}\right)^2-n+1}=\frac{n}{2}-1.
$$
\end{proof}

Note that the value $k=1$ in Proposition \ref{Prop. even n}(i) corresponds to $d=\frac{n}{2}-1$, for which a Hermitian circulant matrix $C$ always exists -- see Proposition~\ref{Prop. trivial}.

The statement of Proposition~\ref{Prop. even n} can be strengthened in the special case when $d$ is integer:

\begin{proposition}\label{Prop. integer d}
Let a Hermitian circulant matrix $C$ of an even order $n$ satisfies \eqref{Conditions} with an integer $d$. Let us denote
\begin{equation}\label{ell}
\ell:=\sqrt{d^2+n-1}.
\end{equation}
Then we have:
\begin{itemize}
\item[(i)] $\ell$ is integer (in other words, $d^2+n-1$ is a perfect square).
\item[(ii)] $\ell\mid \frac{n}{2}$.
\item[(iii)] $\ell\mid (d^2-1)$. In particular, if $d$ is odd, then $\ell\mid \frac{d^2-1}{2}$.
\item[(iv)] If $n-1$ is prime, then $d=\frac{n}{2}-1$.
\end{itemize}
\end{proposition}

\begin{proof}
(i)\; If $d$ is integer, then $\ell=\sqrt{d^2+n-1}$ is obviously either integer or irrational. But $\ell$ cannot be irrational by Proposition \ref{Prop. even n}(i); so $\ell$ is integer.

(ii)\; Since $\ell$ is integer by (i), the statement $\ell\mid \frac{n}{2}$ immediately follows from Proposition \ref{Prop. even n}(i).

(iii)\; Equation \eqref{ell} implies $d^2-1=\ell^2-n$.
Since $\ell\mid n$ by (ii) and obviously $\ell\mid\ell^2$, we have $\ell\mid (d^2-1)$.
In particular, if $d$ is odd, then $d^2+n-1$ is even, so $\ell$ is even. Thus $2\ell\mid \ell^2$. At the same time $2\ell\mid n$ by (ii), so $2\ell\mid(\ell^2-n)$. Hence $2\ell\mid(d^2-1)$ due to $d^2-1=\ell^2-n$.

(iv)\; Statement (i) implies that $n-1=\ell^2-d^2=(\ell-d)(\ell+d)$ for some integer $\ell\geq0$.
If $n-1$ is prime, then necessarily $\ell-d=1$; hence $\ell+d=\ell-d+2d=1+2d$, and so $n-1=1\cdot(1+2d)$. Consequently, $n=2+2d$, thus $d=\frac{n}{2}-1$.
\end{proof}

Proposition~\ref{Prop. integer d} has a series of consequences, which will be formulated below as Corollaries~\ref{conference}, \ref{Hadamard} and \ref{prime}.

\begin{definition}
A complex square matrix $C$ of order $n$ is called a complex \emph{conference matrix} if all its diagonal entries are $0$, its off-diagonal entries are of absolute value $1$, and $CC^*=(n-1)I$.
\end{definition}

\begin{corollary}\label{conference}
A Hermitian circulant conference matrix of an even order $n\geq2$ exists only for $n=2$.
\end{corollary}

\begin{proof}
Let a conference matrix $C$ satisfy the assumptions, i.e., $C$ is a Hermitian circulant matrix obeying \eqref{Conditions} with $d=0$ and an even $n\geq2$. Then Proposition~\ref{Prop. integer d}(iii) implies $\ell\mid-1$; so $|\ell|=1$. Hence, by \eqref{ell}, we have $\sqrt{n-1}=1$; thus $n=2$.
\end{proof}

\begin{definition}
A complex square matrix $H$ of order $n$ is called a complex \emph{Hadamard matrix} if all its entries are of absolute value $1$ and $HH^*=nI$.
\end{definition}

\begin{corollary}\label{Hadamard}
A Hermitian circulant Hadamard matrix of an even order $n\geq2$ exists only if $n$ is a square of an even integer.
\end{corollary}

\begin{proof}
The statement follows immediately from Proposition \ref{Prop. integer d}(i) with $d=1$.
\end{proof}

\begin{remark}\label{rem. CK}
Corollary~\ref{Hadamard} concerns Hermitian circulant complex Hadamard matrices with off-diagonal entries being any complex units (i.e., $|c_j|=1$ for all $j=1,\ldots,n-1$). Let us note that if the off-diagonal entries are restricted to 4th roots of unity, i.e., $c_j\in\{1,-1,\mathrm{i},-\mathrm{i}\}$ for all $j=1,\ldots,n$, it is known that no such matrix of order $n>4$ exists (see Craigen and Kharaghani \cite{CK}). Matrices $C$ with off-diagonal entries from $\{1,-1,\mathrm{i},-\mathrm{i}\}$ and general diagonal values $d\geq0$ will be further discussed in Section~\ref{section 4}.
\end{remark}

\begin{corollary}\label{prime}
Consider a Hermitian circulant matrix $C$ satisfying (\ref{Conditions}) with an integer $d$. If $n/2$ is prime, then $n=2d+2$.
\end{corollary}
\begin{proof}
If $d$ is integer and $n/2=p$ is a prime number, then from Proposition \ref{Prop. integer d}(ii) we have $\ell=1$ or $\ell=p$, where $\ell=\sqrt{d^2+2p-1}$. The case $\ell=1$ cannot occur, as it leads to $1=\sqrt{d^2+2p-1}$, which is impossible for any prime $p$. So $\ell=p=n/2$, which implies $d^2 = \left(\frac{n}{2}-1\right)^2$; hence $d=\frac{n}{2}-1$.
\end{proof}

Now we will extend the result of Corollary~\ref{prime} in two ways (Propositions \ref{two primes} and \ref{prime power}).

\begin{proposition}\label{two primes}
Let $C$ be a Hermitian circulant matrix satisfying \eqref{Conditions} with an integer $d$. If $n/2$ is a product of two primes, then $n=2d+2$.
\end{proposition}

\begin{proof}
Let $n/2=pq$ for $p,q$ being primes. Then \eqref{ell} gives
\begin{equation}\label{ell primes}
\ell=\sqrt{d^2+2pq-1}.
\end{equation}
We may assume $p\leq q$ without loss of generality. Proposition~\ref{Prop. integer d}(ii) gives $\ell\mid pq$; hence $\ell\in\{1,p,q,pq\}$. If $\ell=pq$, equation~\eqref{ell primes} leads to $d=pq-1=\frac{n}{2}-1$, i.e., $n=2d+2$. Case $\ell=1$ cannot occur, because the right hand side of \eqref{ell primes} is greater than $1$ for any two primes $p,q$. Similarly, $\ell=q$ is not possible, because \eqref{ell primes} gives $d=\sqrt{q^2-2pq+1}$, where $q^2-2pq+1\leq1-q^2<0$ due to the assumption $p\leq q$. In the following we demonstrate that the case $\ell=p$ is impossible as well. If $\ell=p$, equation \eqref{ell primes} implies $p^2=d^2+2pq-1$, so
$$
d^2=(p-q)^2-q^2+1.
$$
Hence $p-q>d$. Denoting the difference $p-q-d$ by an integer variable $k>0$, one can rewrite the last equation as follows:
$$
(q+k+1)(q+k-1)=2kp.
$$
Thus $2k\mid(q+k+1)(q+k-1)$. Let $2k=a\cdot b$ for $a\mid(q+k+1)$ and $b\mid(q+k-1)$. Then
$$
p=\frac{q+k+1}{a}\cdot\frac{q+k-1}{b}.
$$
Since $p$ is prime, necessarily $\frac{q+k+1}{a}=1$ or $\frac{q+k-1}{b}=1$.
\begin{itemize}
\item If $\frac{q+k+1}{a}=1$, we have $b=\frac{2k}{a}=\frac{2k}{q+k+1}<2$; thus $b=1$. Therefore, $q+k-1=1$, which can never be true for an integer $k>0$ and prime $q$.
\item If $\frac{q+k-1}{b}=1$, we have $a=\frac{2k}{b}=\frac{2k}{q+k-1}\leq\frac{2k}{1+k}<2$; hence $a=1$. Thus $q+k+1=1$, which is again false for any positive integer $k>0$ and prime $q$.
\end{itemize}
\end{proof}

In the following proposition, we omit the cases $d=0$ and $d=1$ that were treated generally in Corollaries \ref{conference} and \ref{Hadamard}.

\begin{proposition}\label{prime power}
Let $C$ be a Hermitian circulant matrix satisfying \eqref{Conditions} with an integer $d\geq2$. If $n/2$ is a power of a prime, then $n=2d+2$.
\end{proposition}

\begin{proof}
Let $n/2=p^m$ for $p$ being prime and $m$ being a non-negative integer.
Proposition~\ref{Prop. integer d}(ii) implies that $\ell\mid p^m$; hence $\ell=p^j$ for $0\leq j\leq m$. By \eqref{ell} we have $\ell=\sqrt{d^2+2p^m-1}$; i.e., $p^{2j}=d^2+2p^m-1$.
Consequently,
\begin{equation}\label{equation power}
p^{2j}-2p^m=d^2-1.
\end{equation}
Note that $d^2-1>0$ because of the assumption $d\geq2$.
Now we distinguish $p=2$ and $p\geq3$.

\begin{itemize}
\item If $p=2$, the left hand side of equation \eqref{equation power} takes the form
\begin{equation}\label{LHS 2}
2^{2j}-2\cdot2^m=2^{2j}-2^{m+1}=2^{m+1}(2^{2j-m-1}-1);
\end{equation}
so \eqref{equation power} with its right hand side $d^2-1>0$ can be satisfied only if $2j-m-1>0$ and $2^{m+1}\mid(d-1)(d+1)$. Thus both $d-1$ and $d+1$ are even. Since one of any two consecutive even numbers must be oddly even, we have two possibilities:
$$
(d-1=2a \;\wedge\; d+1=2^m b) \quad\text{or}\quad (d-1=2^m a \;\wedge\; d+1=2b),
$$
where both $a,b$ are odd.
So $2^m\mid(d+1)$ or $2^m\mid(d-1)$. Thus in either case we have $2^m\leq d+1$. Recalling that $2^m=\frac{n}{2}$, we obtain $\frac{n}{2}\leq d+1$. This inequality together with Proposition \ref{Prop. even n}(iii) gives $n=2d+2$.

\item If $p\geq 3$, we rewrite \eqref{equation power} as follows:
\begin{equation}\label{LHS 3}
p^m(p^{2j-m}-2)=(d-1)(d+1).
\end{equation}
Since $d-1$ and $d+1$ obviously cannot be both divisible by $p\geq3$, we infer that
$$
p^m\mid(d-1) \quad\text{or}\quad p^m\mid(d+1).
$$
Therefore, in either case we have $p^m\leq d+1$. Since $p^m=\frac{n}{2}$, we have $\frac{n}{2}\leq d+1$. Together with Proposition \ref{Prop. even n}(iii), we get $n=2d+2$.
\end{itemize}
\end{proof}

\begin{remark}
Let us emphasize that Proposition~\ref{Prop. integer d}, Corollary~\ref{prime} and Propositions \ref{two primes} and \ref{prime power} concern circulant matrices $C$ obeying \eqref{Conditions} with an \emph{integer} on the main diagonal. They have no implications on matrices $C$ with non-integer values $d$.
\end{remark}

\begin{example}
Let us consider various integer values of $d\geq2$, for which we will find necessary conditions on $n$ using Proposition~\ref{Prop. integer d}.

\begin{itemize}
\item $d=2$: $n+3$ must be a square and $\sqrt{n+3}\mid 3$. The only even solution is $n=6$.
\item $d=3$: $n+8$ must be a square and $\sqrt{n+8}\mid 4$. The only even solution is $n=8$.
\item $d=4$: $n+15$ must be a square and $\sqrt{n+15}\mid 15$. The only even solutions are $n=10$ and $n=210$.
\item $d=5$: $n+24$ must be a square and $\sqrt{n+24}\mid 12$. The only even solutions are $n=12$ and $n=120$.
\end{itemize}
Proposition~\ref{Prop. trivial} implies that the matrices $C$ of orders $n=2d+2$ (i.e., those with $(d,n)\in\{(2,6),(3,8),(4,10),(5,12)\}$) exist. Recall that a matrix $C = \circ_{2d+2}(d,-1,-1,\ldots,-1)$ obeys \eqref{Conditions}. On the other hand, the existence of matrices $C$ with $(d,n)\in\{(4,210),(5,120)\}$ is open.
\end{example}

\begin{example}\label{example even n}
In this example we shall consider various even values of $n$, for which we will find necessary conditions on $d$ using Proposition~\ref{Prop. even n}.

\begin{itemize}
\item $n\leq14$: The condition $k\leq\frac{n}{2\sqrt{n-1}}$ from Proposition \ref{Prop. even n}(i) gives $k<2$. Hence $k=1$, so $d=\frac{n}{2}-1$ is the only possible value of $d$.
\item $n=16$: The condition $k\leq\frac{n}{2\sqrt{n-1}}=\frac{16}{2\sqrt{15}}$ implies $k=1$ or $k=2$. The value $k=1$ gives the trivial solution $d=\frac{n}{2}-1=7$. Equation \eqref{d2} with the value $k=2$ leads to $d=1$; this case is disproved by numerical simulations (see Appendix~\ref{appendix} where all possible diagonal values $d<\frac{n}{2}-1$ for matrices $C$ of orders $n$ up to $n=22$ are listed).
\item $n=18$: $k\leq\frac{n}{2\sqrt{n-1}}=\frac{18}{2\sqrt{17}}$ implies $k=1$ or $k=2$. The value $k=1$ gives the trivial solution $d=\frac{n}{2}-1=8$. Equation \eqref{d2} with the value $k=2$ leads to $d=\frac{\sqrt{13}}{2}\notin\mathbb{Q}$, which is impossible by Proposition \ref{Prop. even n}(ii).
\item $n=20$: $k\leq\frac{n}{2\sqrt{n-1}}=\frac{20}{2\sqrt{19}}$ implies $k=1$ or $k=2$. The value $k=1$ gives $d=\frac{n}{2}-1=9$, $k=2$ leads to $d=\sqrt{6}\notin\mathbb{Q}$. So the only possible value of $d$ is $d=9$.
\item $n\in\{22,24,\ldots,100\}$: Similarly as above, one obtains mostly either trivial solutions $d=\frac{n}{2}-1$ or forbidden values $d\notin\mathbb{Q}$, with the following $9$ exceptions:

\begin{center}
\begin{tabular}{|c|c||c|c||c|c|}
\hline
$n$ & $d$ & $n$ & $d$ & $n$ & $d$ \\
\hline
$36$ & $1$ & $64$ & $1$ & $78$ & $17/4$ \\ \hline
$40$ & $7/3$ & $66$ & $7/4$ & $96$ & $7$ \\ \hline
$56$ & $17/3$ & $70$ & $11/4$ & $100$ & $1$ \\ \hline
\end{tabular}
\end{center}

The existence of matrices $C$ for the $9$ combinations of $n$ and $d$ in the above table is open.
\end{itemize}
\end{example}

\subsection{Matrices $C$ of odd orders}\label{odd orders}

In case of odd $n$, as well as in case of those even $n$ that cannot be completely resolved using tools from Section~\ref{even orders}, we searched for allowed values of $d$ numerically using the following idea. Since $C$ is Hermitian and satisfies $C^2=(d^2+n-1)I$, the eigenvalues of $C$ are $\pm\sqrt{d^2+n-1}$. Taking vectors $(\lambda_0,\lambda_1,\ldots,\lambda_{n-1})$ with entries $\pm\sqrt{d^2+n-1}$, we calculated the terms of the corresponding generator $(c_0,c_1,\ldots,c_{n-1})$ using formula \eqref{invDFT}. Then we checked whether the values $c_j$ obey conditions \eqref{Conditions}, i.e., $|c_j|=1$ for all $j=1,\ldots,n-1$. In this way we found all allowed values of $d$ that satisfy conditions \eqref{Conditions} up to $n=22$. The results for odd orders $n$ are summarized in Table~\ref{tabla}.
\begin{table}[h!]
\begin{center}
{\renewcommand{\arraystretch}{1.7}
\begin{tabular}{ |c|c||c|c| }
\hline
$n$ & $d$ & $n$ & $d$ \\ 
\hline
$3$ & $\frac{1}{2}$ &  $13$ & $\frac{11}{2 }$, $\frac{5}{2\sqrt{3} }$ \\ \hline
$5$ & $\frac{3}{2}$ &  $15$ & $\frac{13}{2}$, $\frac{1}{4}$ \\ \hline
$7$ & $\frac{5}{2}$, $\frac{1}{2 \sqrt{2}}$  & $17$ & $\frac{15}{2}$  \\ \hline
$9$ & $\frac{7}{2}$ &  $19$ & $\frac{17}{2}$, $\frac{1}{2 \sqrt{5}}$\\ \hline
$11$ & $\frac{9}{2}$,  $\frac{1}{2 \sqrt{3}}$  &   $21$ & $\frac{19}{2}$, $\frac{11}{4}$ \\ \hline

\hline
\end{tabular}
\caption{Numerical results for odd order $n$. }
\label{tabla} 
}
\end{center}
\end{table}

The main difference between the even and odd order $n$ is that for even $n$ diagonal values $d$ have to be rational, whereas they can be irrational for odd $n$. In Appendix~\ref{appendix} we present examples of the generators associated to values $n$ from $2$ to $22$ with the main diagonal $d$ different from $\frac{n}{2}-1$ which were found by the method described above. (Recall that a solution with $d=\frac{n}{2}-1$ always exists -- see Proposition~\ref{Prop. trivial}.)

\subsection{Matrices $C$ with off-diagonal entries from $\{1,-1,\mathrm{i},-\mathrm{i}\}$}\label{section 4}

Let us now discuss a special case of complex circulant matrices satisfying \eqref{Conditions} with off-diagonal entries being 4th roots of unity, i.e., $c_j\in\{1,-1,\mathrm{i},-\mathrm{i}\}$ for all $j=1,\ldots,n-1$. As recalled in Remark~\ref{rem. CK}, Craigen and Kharaghani proved that Hadamard matrices (satisfying \eqref{Conditions} for $d=1$) of this type exist only of order $n=4$ (and trivially $n=1$). In this section we extend their result to any $d$ that is not an odd integer, showing that the order of such matrix $C$ is necessarily $n=2d+2$. Furthermore, if a generalization of the circulant Hadamard conjecture proposed in \cite{TG} is true, the necessary condition $n=2d+2$ applies on matrices $C$ with odd diagonal values $d\geq0$ as well.

\begin{theorem}\label{4th roots}
If $d\geq0$ is not an odd integer, then a Hermitian circulant matrix $C=\circ_n(c_0,c_1,\ldots,c_{n-1})$ ($n\geq2$) satisfying \eqref{Conditions} with off-diagonal entries $c_j\in\{1,-1,\mathrm{i},-\mathrm{i}\}$ exists only of order $n=2d+2$. Moreover, $C$ is real and takes the form
\begin{itemize}
\item $C=\circ_{2d+2}(d,-1,-1,\ldots,-1)$ or $C=\circ_{2d+2}(d,1,-1,1,-1,1,\ldots,-1,1)$ for even $d$;
\item $C=\circ_{2d+2}(d,-1,-1,\ldots,-1)$ for $d$ being half-integer.
\end{itemize}
\end{theorem}

\begin{proof}
We start the proof similarly as Craigen and Kharaghani in \cite[Thm.~7]{CK}. Let us write the circulant matrix $C$ satisfying the assumptions as $C=A+\mathrm{i}B$, where $A,B$ are real matrices. Then both $A$ and $B$ are circulant matrices, $A$ is symmetric, $B$ is skew-symmetric. Let us denote $M=A+B$; then $M$ is a circulant matrix satisfying $MM^T=(A+B)(A-B)=A^2-B^2$ (recall that $A,B$ are circulant matrices, so they commute). Since $CC^*=(A+\mathrm{i}B)(A+\mathrm{i}B)=A^2-B^2+2\mathrm{i}AB=(d^2+n-1)I$, we have $AB=0$ and $A^2-B^2=(d^2+n-1)I$. So $MM^T=(d^2+n-1)I$. To sum up, $M$ is a real circulant matrix satisfying \eqref{Conditions}. Now, taking advantage of results of \cite{TG}, we have:
\begin{itemize}
\item $2d$ must be integer; thus $M$ exists only for $d$ being half-integer or integer (and so does $C$) \cite[Prop.~3.1]{TG};
\item if $d$ is half-integer, then $n=2d+2$ \cite[Prop.~3.1]{TG} and $M=\circ_{2d+2}(d,-1,-1,\ldots,-1)$ \cite[Rem.~3.2]{TG};
\item if $d$ is even integer, then $n=2d+2$ \cite[Thm.~3.5]{TG} and $M$ is symmetric \cite[Prop.~3.4]{TG}. Moreover, the value $n$ is oddly even, thus \cite[Sect.~5]{TG} implies that $M$ has one of the forms
\begin{align*}
M&=\circ_{n}\left(\frac{n}{2}-1,-1,-1,\ldots,-1\right)=\circ_{2d+2}(d,-1,-1,\ldots,-1), \\
M&=\circ_{n}\left(\frac{n}{2}-1,1,-1,1,-1,1,\ldots,-1,1\right)=\circ_{2d+2}(d,1,-1,1,-1,1,\ldots,-1,1).
\end{align*}
\end{itemize}
Finally, since $M$ is symmetric in all cases, we conclude that $B=0$; hence $C=A=M$ is real.
\end{proof}

\begin{remark}
It was conjectured in \cite[Conjecture 3.6]{TG} that real circulant matrices of order $n\geq2$ satisfying \eqref{Conditions} with odd values $d>0$ exist only for $n=2d+2$ as well. This is a generalization of the \emph{circulant Hadamard conjecture} stating that there a real circulant Hadamard matrix exists only of order $n=4$ (and trivially of order $n=1$). If the generalized conjecture is true, one can extend the argument in the proof of Theorem~\ref{4th roots} to odd $d>0$ as well, obtaining that a Hermitian circulant matrix $C$ satisfying \eqref{Conditions} with an odd $d>0$ and off-diagonal entries $c_j\in\{1,-1,\mathrm{i},-\mathrm{i}\}$ exists only of order $n=2d+2$. Moreover, we know from \cite[Sect.~5]{TG} that any real circulant matrix $M$ satisfying \eqref{Conditions} with an odd $d$ and $n=2d+2$ (so $n$ is a multiple of $4$) has one of the forms
\begin{gather*}
\circ_{2d+2}(d,-1,-1,\ldots,-1),\quad \circ_{2d+2}(d,1,-1,1,-1,1,\ldots,-1,1), \\
\circ_{2d+2}(d,1,1,-1,-1,1,1,-1,-1,\ldots,1,1,-1),\quad
\circ_{2d+2}(d,-1,1,1,-1,-1,1,1,-1,\ldots,-1,1,1).
\end{gather*}
Hence we obtain, similarly as in the proof of Theorem~\ref{4th roots}, that the matrix $C$ can be either real, taking one of the forms
$$
\circ_{2d+2}(d,-1,-1,\ldots,-1),\quad \circ_{2d+2}(d,1,-1,1,-1,1,\ldots,-1,1),
$$
or complex taking one of the forms
$$
\circ_{2d+2}(d,\mathrm{i},1,-\mathrm{i},-1,\mathrm{i},1,-\mathrm{i},-1,\ldots,\mathrm{i},1,-\mathrm{i}),\quad
\circ_{2d+2}(d,-\mathrm{i},1,\mathrm{i},-1,-\mathrm{i},1,\mathrm{i},-1,\ldots,-\mathrm{i},1,\mathrm{i})
$$
(notice that the last two matrices are conjugate transposes to each other).
\end{remark}

\section{Circulant matrices over $\mathbb{Z}_m$}

In this section we will briefly consider circulant matrices $C$ satifying conditions \eqref{Conditions} with entries $c_j$ being elements of the ring $\mathbb{Z}_m=\{0,1,\ldots,m-1\}$ for some $m$. 
In this particular case, the condition $|c_j|=1$ is meant as $c_j \equiv 1 \pmod m \text{ or } c_j \equiv -1 \equiv m-1 \pmod m.$ 

First of all, note the following fact:
\begin{remark}
For any $C=\circ_n (d,c_1,c_2,\dots, c_{n-1})$ over $\mathbb Z_m$ such that  $C\cdot C^T=(d^2+n-1)I,$ the matrix $-C=\circ_n (m-d,-c_1,-c_2,\dots, -c_{n-1})$ fulfills the same condition.
\end{remark}

\begin{proposition}\label{Zuza even m}
Let $C=\circ_n (d,c_1,c_2,\dots, c_{n-1})$ be defined over $\mathbb{Z}_m$ with $c_i\equiv\pm 1 \pmod m.$  If $m$ is even, then $n$ is also even.
\end{proposition}

\begin{proof}
The dot product of any two  distinct rows is sum of $n$ terms, of which $n-2$ are equal to $\pm 1$ and the other two are $\pm d.$ Thus the sum is even if $n$ is even and odd if $n$ is odd. Hence if $m$ is even, $n$ must be too.
\end{proof}

\begin{remark}
The converse implication of Proposition~\ref{Zuza even m} does not hold. For example, consider $C=\circ_4(2,1,1,1)$ over $\mathbb Z_3.$ In this case $n=4$ is even and $m=3$ is odd, and for this matrix $C\cdot C^T=I.$ 
\end{remark}

\subsection{Odd $n$}

If a circulant matrix $C$ is defined over $\mathbb R$, satisfies the conditions \eqref{Conditions}  and its order $n$ is odd, then the generator of $C$ is $(\frac{n}{2}-1,-1,\dots, -1),$ so the matrix $C$ has to be symmetric. This follows from \cite[Prop.~3.1 and Sect.~5]{TG}.

If $C$ of an odd order $n$ is defined over $\mathbb Z_m$, the situation is different. There exist non-symmetric matrices $C$ satisfying conditions \eqref{Conditions}. Consider for example the matrix $C=\circ_9 (1,1,1,1,1,1,1,1,-1)$ over $\mathbb{Z}_5$ is not symmetric and satisfies \eqref{Conditions}.

\subsection{Symmetric matrices}

If $C$ is a symmetric matrix over $\mathbb Z_m$, i.e., $c_k=c_{n-k}$,  where the subscripts are interpreted modulo $n,$
the condition  $C\cdot C^T=(d^2+n-1)I$ leads to
$$2c_0 c_{k}+\sum_{j=1,\\ j\neq k}^{n-1} c_j\cdot c_{n-k+j}\equiv 0 \pmod m$$
for all $k=1,\dots, n-1$ (subscripts are interpreted modulo $n$), 
i.e.,
$$2d c_{k}+\sum_{j=1,\\ j\neq k}^{n-1} c_j\cdot c_{n-k+j}\equiv 0 \pmod m.$$
Note that due to the symmetry, it is sufficient to  verify this condition for $k=1,\ldots,\lceil\frac{n}{2}\rceil+1$.
\vskip0.5cm

\subsubsection{Matrix with the generator $(d,-1,-1,\ldots,-1)$}\label{section trivial}

In analogy with Proposition~\ref{Prop. trivial}, we can formulate the following statement:

\begin{proposition}
A symmetric circulant matrix $C$ over $\mathbb{Z}_m$ satisfying \eqref{Conditions} exists for each $n$ and $d$ such that $n\equiv 2d+2 \mod m.$
\end{proposition}

\begin{proof}
Consider the matrix $C=\circ_n(d,-1,-1,\dots,-1)$ over $\mathbb Z_m.$ The condition $C\cdot C^T=(d^2+n-1)I$ leads to
$$2d-\sum_{j=1,\\ j\neq k}^{n-1} 1\equiv 0 \pmod m,$$
i.e.,
$$2d-n+2\equiv 0 \pmod m,$$
which is equivalent to $n\equiv 2d+2 \pmod m.$
\end{proof}

\begin{example}
In the special case when $n=m+2$ and $m$ is odd, any matrix of the type
 $$\circ_{m+2}(0,-1,\dots,-1)$$
over $\mathbb{Z}_m$ fulfills the condition \eqref{Conditions}. If $n=m+2$ and $m$ is even, then 
any matrix of the type
 $$\circ_{m+2}(0,-1,\dots,-1),$$
$$\circ_{m+2}(\frac{m}{2},-1,\dots,-1)$$
over $\mathbb{Z}_m$ fulfills the condition \eqref{Conditions}.

\end{example}

\begin{example}\label{m=2}
If $m=2$, i.e., for $C$ defined over $\mathbb Z_2,$ the congruence
$$2d-n+2\equiv 0 \mod 2$$
 is trivially fulfilled for any even $n$ and any $d.$
So both matrices
$$\circ_{2k}(0,-1,\dots,-1),$$
$$\circ_{2k}(1,-1,\dots,-1)$$
over $\mathbb Z_2$ fulfill the conditions \eqref{Conditions}.
A matrix $C$ over $\mathbb{Z}_2$ of an odd order $n$ satisfying \eqref{Conditions}does not exist, in keeping with Proposition~\ref{Zuza even m}.

\end{example}

\begin{remark}
Because $1\equiv -1 \pmod 2,$ for matrices over $\mathbb Z_2$ there is no difference between $c_i=1$ and $c_i=-1.$ Therefore, Example~\ref{m=2} implies that any circulant matrix $C$ of an even order $n$ over $\mathbb{Z}_2$ with off-diagonal entries $\pm1$ obeys conditions \eqref{Conditions}.
\end{remark}

\subsubsection{An example of a matrix $C$ that does not fulfill the conditions over $\mathbb R$ but fulfills them over $\mathbb Z_m$ }

By \cite{TG}, a symmetric circulant matrix $C$ over $\mathbb R$ satisfies conditions \eqref{Conditions} only if its generator is $(d,-1,-1, \dots,-1)$ or $(d,-1,+1,-1,+1,\dots,-1).$
Let us demonstrate that this necessary condition does not extend to matrices $C$ defined over $\mathbb{Z}_m$.
We will construct an example of a symmetric circulant matrix $C$ that does not fulfills the conditions \eqref{Conditions} over $\mathbb R$, but fulfills them over $\mathbb Z_m$.
 
Let $C=\circ_n(d,c_1,c_2,\dots,c_{n-1})$ be defined over $\mathbb Z_m,$ where $n$ is even, $c_{\frac{n}{2}}=1 \pmod m$ and $c_i=-1 \pmod m$ for all $i\neq \frac{n}{2}.$ I.e., the generator of $C$ is
\begin{equation}\label{C one plus}
(d,\underbrace{-1,\ldots,-1}_{\frac{n}{2}-1 \text{ terms}},1,\underbrace{-1,\ldots,-1}_{\frac{n}{2}-1 \text{ terms}}).
\end{equation}
The dot product of the 0-th row and the $k-$th row of $C$ is
\begin{equation}\label{Zuza4}
2d c_{k}+\sum_{j=1,\\ j\neq k}^{n-1} c_j\cdot c_{n-k+j},  
\end{equation}
where the subscripts are interpreted modulo $n.$
So for $k\neq \frac{n}{2},$ \eqref{Zuza4} is equal to
$2d-(n-4)+2,$
and for
$k= \frac{n}{2},$ \eqref{Zuza4} gives
$2d-(n-2).$ 

Therefore, the condition $C\cdot C^T=(d^2+n-1)I$ requires the following two congruences to be fulfilled:
\begin{equation}\label{Zuza5}
2d\equiv -n+2 \pmod m \qquad\wedge\qquad
2d\equiv n-6 \pmod m.
\end{equation}
In examples below, we will consider explicit solutions.

\vskip0.5cm

\begin{example}
Let $C$ of an even order $n$ defined over $\mathbb Z_m$ satisfy \eqref{C one plus} and $m\mid n.$ Then from the congruences we have $0\equiv 8 \mod m,$
so $m=2,4,\text{or} \, 8.$
We will describe each situation separately.
\begin{enumerate}
\item For $m=2,$ already examined in Section~\ref{section trivial}, 
any matrix $C$ over $\mathbb Z_2$ with even $n$ such that ($m\mid n$) of the type
$$\circ_{2k}(d,-1,\dots,-1,1,-1, \dots,-1)=\circ_{2k}(d,1,\dots,1,1,1, \dots,1)$$
with $k\in\mathbb{N}$ fulfills the conditions \eqref{Conditions}.

\item For $m=4,$ we get $d=1\, \text{or}\, 3,$ so  
 any matrix $C$ over $\mathbb Z_4$ of the type
$$\circ_{4k}(1,-1,\dots,-1,1,-1, \dots,-1)$$
$$\circ_{4k}(3,-1,\dots,-1,1,-1, \dots,-1)$$
with $k\in\mathbb{N}$ fulfills the conditions \eqref{Conditions}.

\item  For $m=8:$ we have $d=1\, \text{or}\, 5 \mod 8,$ so  
any matrix $C$ over $\mathbb Z_8$ of the type
$$\circ_{8k}(1,-1,\dots,-1,1,-1, \dots,-1)$$
$$\circ_{8k}(5,-1,\dots,-1,1,-1, \dots,-1)$$
with $k\in\mathbb{N}$ fulfills the conditions \eqref{Conditions}.
\end{enumerate}

	\end{example}
\vskip0.5cm

\begin{example}
Let $C$ of an even order $n=2k$ defined over $\mathbb Z_m$ satisfy \eqref{C one plus} and let $m$ be odd. Then the congruences \eqref{Zuza5} lead to
$$2d\equiv -2k +2 \pmod m \quad\wedge\quad 2d\equiv 2k-6 \pmod m.$$
Dividing in both congruences by $2$ (which is a correct step due to $\gcd(2,m)=1$), we get
$$ d\equiv -k+1 \pmod m \quad\wedge\quad d \equiv -3+k \pmod m,$$
hence
$$
2d \equiv -2 \pmod m,
$$
and so
$$d\equiv -1 \pmod m.$$
In this case we get the matrices $C$ over $\mathbb Z_m$ of the type

$$\circ_{2m\ell+4}(m-1,-1,\dots,-1,1,-1, \dots,-1),$$
where $m$ is odd.
\end{example}

\vskip0.5cm

\begin{example}
Let $C$ of an even order $n=2k$ defined over $\mathbb Z_m$ satisfy \eqref{C one plus} and let $m$ be even. Then the congruences \eqref{Zuza5} lead to
$$ \frac{n}{2}\equiv -d+1 \pmod {\frac{m}{2}} \quad\wedge\quad \frac{n}{2}\equiv 3+d \pmod {\frac{m}{2}}.$$
By adding/substracting these two congruences, we get
$$n\equiv 4 \pmod {\frac{m}{2}},$$
$$2d\equiv -2 \pmod {\frac{m}{2}}.$$
Hence, for odd  $ \frac{m}{2},$ we obtain the matrices of type

$$\circ_{m\ell+4} (d,-1,\dots,-1,1,-1,\dots,-1),$$
where $\ell\in\{0,1,2,\ldots\}$ and
$d\equiv -1 \pmod {\frac{m}{2}}.$

For even  $ \frac{m}{2},$ we obtain the matrices of type

$$\circ_{\frac{m\ell}{2}+4} (d,-1,\dots,-1,1,-1,\dots,-1),$$
where $\ell\in\{0,1,2,\ldots\}$ and
$d\equiv -1 \pmod {\frac{m}{4}}.$

\end{example}

\bigskip

\section{Application: Mutually unbiased bases}

In this section, we present an application for the particular case of circulant matrices $C$ satisfying conditions (\ref{Conditions}) with $d=1$. If $c_0=d=1$, then all the entries of the generator have absolute value $1$, so $C = \circ_n(c_0,c_1,\ldots,c_{n-1})$ defines an unormalized circulant complex Hadamard matrix of order $n$. Here, the absence of normalization is in the sense that matrix $C$ is proportional to unitary but not unitary, as stated in Conditions (\ref{Conditions}). From now on, we consider normalized circulant matrices $C$, which differ from those defined in (\ref{Conditions}) by a constant factor $1/\sqrt{n}$. The reason to introduce this normalization is because columns of the considered unitary matrices define mutually unbiased -orthonormal- bases of $\mathbb{C}^n$. Throughout this section, we assume that $C$ is \emph{not necessarily Hermitian}; non-Hermitian matrices $C$ are allowed. The problem of the existence of matrices $C$ having constant diagonal $d=1$ is particularly hard to solve in its full generality for arbitrary large $n$, as it contains the long-standing \emph{circulant Hadamard conjecture} \cite{R63}.

Let $\lambda=(\lambda_0,\dots,\lambda_{n-1})$ be the vector of the eigenvalues of $C$. From Eqs.(\ref{Conditions}) we know that $\lambda_j=\sqrt{d^2+n-1}\,e^{i\alpha_j}$, where $\alpha_j\in[0,2\pi]$ are suitable phases, for every $j=0,\dots,n-1$. As a basic property of circulant matrices, the generator of $C$ is given by $g=F\lambda$, where $F$ is the discrete Fourier transform of order $n$. 

In order to satisfy conditions (\ref{Conditions}) we should have $[g]_0=d$ and $|[g]_j|=1$, for every $j=0,\dots,n-1$, where $[g]_k$ denotes the $k$th entry of vector $g$. Let us now show that this particular problem for $d=1$ is one-to-one related to a well-known problem in quantum information theory: the \emph{mutually unbiased bases} problem.

Two orthonormal bases in $\mathbb{C}^n$, $\{\phi_j\}_{j=0,\dots,n-1}$ and $\{\psi_k\}_{k=0,\dots,n-1}$, are mutually unbiased (MU) if $|\langle\phi_j|\psi_k\rangle|^2=\frac{1}{n}$, for every $j,k=0,\dots,n-1$. Two MU bases exist in every dimension $n\geq2$. Indeed, the canonical basis in dimension $n$ is MU to the basis defined by the columns of the discrete Fourier transform for any order $n\geq2$.

Even more, three pairwise MU bases (MUB) exist in every dimension $n\geq2$ \cite{BBRV02}. They are given by the eigenvectors bases of the three unitary operators $Z,X$ and $XZ$, where $Z=\sum_{j=0}^{n-1}\omega^j\langle e_j,\cdot\rangle\,e_j$, $X=\sum_{j=0}^{n-1}\langle e_j,\cdot\rangle\,e_{j+1 \pmod {n}}$. Here, $\{e_j\}_{j=0,\dots,n-1}$ denotes the $j$th element of the canonical basis and $\omega=\mathrm{e}^{2\pi \mathrm{i}/n}$. Eigenvectors of $Z$ are given by the canonical basis, whereas the colums (or rows) of the discrete Fourier transform of order $n$ are eigenvectors of $X$. For prime values of $n$, the eigenvectors basis of the product operator $XZ$ is given by 
\begin{equation}
\varphi_j=\frac{1}{\sqrt{n}}\sum_{k=0}^{n-1}\omega^{-jk-s_k},
\end{equation}
where $s_k=k+\cdots+n-1$, cf. Eq.(3) in Ref. \cite{BBRV02} for the special case $k=1$ (here $k$ follows notation used in Ref. \cite{BBRV02}).

In general, there are at most $n+1$ MUB in dimension $n$, where the upper bound can be saturated for every prime \cite{I81} and prime power \cite{WF89} dimension $n$. For any other composite dimension, e.g. $n=6$, it is not known how many pairwise MU bases can be constructed; this question is one of the main open problems in quantum information theory. The importance of MU bases relies on the fact that two physical observables, represented by hermitian operators, are canonical (i.e. as different as possible) if and only if their eigenvectors bases are MU. So, translated to physics, the open question is about how many mutually canonical observables exist in every finite dimension $n$. Furthermore, the existence of a maximal set of $n+1$ MUB in dimension $n$ provides a protocol for quantum state reconstruction from experimental measurements \cite{I81}, which maximizes the robustness of reconstruction under the presence of errors in both state preparation and measurement stages \cite{S06}. \medskip

Before introducing the relation to our problem let us establish a standard notation. When refering to a set of $m$ MU bases we will use the notation $\{M_1,\dots,M_m\}$, where $M_j$, $j=1,\dots,m$, are unitary matrices containing the vectors forming the bases in its columns. According to this notation, note that $M_j^*M_k=nH^{(j,k)}$, where all matrices $H^{(j,k)}$ are unnormalized complex Hadamard matrix. For instance, $\{I,F\}$, i.e. identity and Fourier matrices, define a pair of MU bases for any order $n$.\bigskip

\begin{proposition}
The identity matrix $I$ together with discrete Fourier transform $F$ and any circulant matrix $C$ satisfying conditions (\ref{Conditions}) with $d=1$ define a set of three MUB in dimension $n\geq2$. 
\end{proposition}
\begin{proof}
Identity matrix $I$ is MUB to both matrices $F$ and $C$ of order $n$ because every entry of these two matrices has the same amplitude $1/\sqrt{n}$. Also, eigenvalues of $C$ obey Eq.(\ref{eigenvalues}), which immediately imply that $F$ and $C$ are MUB. This is so because $C$ is a unitary matrix, so it has $n$ unimodular complex eigenvalues.
\end{proof}\bigskip

Let us illustrate the above result with the explicit solution for a maximal set of MUB in dimensions $d=2$ and $d=3$, where three and four MUB exist, respectively:
\begin{equation}
I=\left(\begin{array}{ccc}
1&0\\
0&1\\
\end{array}\right)\hspace{0.3cm}
F=\frac{1}{\sqrt{2}}\left(\begin{array}{ccc}
1&1\\
1&-1\\
\end{array}\right)\hspace{0.3cm}
C=\frac{1}{\sqrt{2}}\left(\begin{array}{ccc}
1&i\\
i&1
\end{array}\right)\hspace{0.3cm}
\end{equation}
and
\begin{equation}
I=\left(\begin{array}{ccc}
1&0&0\\
0&1&0\\
0&0&1
\end{array}\right)\hspace{0.3cm}
F=\frac{1}{\sqrt{3}}\left(\begin{array}{ccc}
1&1&1\\
1&\omega&\omega^2\\
1&\omega^2&\omega
\end{array}\right)\hspace{0.3cm}
C_1=\frac{1}{\sqrt{3}}\left(\begin{array}{ccc}
\omega&1&1\\
1&\omega&1\\
1&1&\omega
\end{array}\right)\hspace{0.3cm}
C_2=\frac{1}{\sqrt{3}}\left(\begin{array}{ccc}
\omega^2&1&1\\
1&\omega^2&1\\
1&1&\omega^2
\end{array}\right)
\end{equation}\bigskip

Le us mention that the discrete Fourier transform $F$ of prime order $n$ is equivalent to a circulant matrix $C$ satisfying conditions \eqref{Conditions} with $d=1$ \cite{B89,F01}. Here, we consider the following notion of equivalence: two matrices $A$ and $B$ are equivalent if there exists diagonal unitary matrices $D_1,D_2$ and permutation matrices $P_1,P_2$ such that $A=D_1P_1BP_2D_2$. Furthermore, any circulant matrix $C$ satisfying conditions \eqref{Conditions}, with $d=1$ and prime order $n$, is equivalent to $F$  (cf. Theorem 1.2 in Ref. \cite{HS16}).  \bigskip

\section{Acknowledgements}  
DG and DU kindly acknowledge support from Grant FONDECYT Iniciaci\'{o}n number 11180474, Chile.\\ 
DU acknowledges support from the project ANT1856, Universidad de Antofagasta.\\
OT acknowledges support from the Czech Science Foundation (GA\v{C}R) within the project 17-01706S.

\appendix

\appendixpage

\section{Computer simulations}\label{appendix}

In this section we present all combinations of $n$ and $d$ with $n=2,\ldots,22$ and $d\neq\frac{n}{2}-1$ such that a circulant matrix $C$ satisfying conditions \eqref{Conditions} exists. They were obtained using the algorithm described in Section~\ref{odd orders}. Each such pair $(n,d)$ is supplemented with an example of a generator of $C$. Notice that the results of the simulations disprove the existence of $C$ with $d\neq\frac{n}{2}-1$ for all even values $n\leq22$, in particular for $n=16$ (discussed in Example~\ref{example even n}).

\begin{itemize}

\item $n = 7$, $d = \frac{1}{2\sqrt{2}}$\\
\vspace{-0.2cm}
\begin{equation}
\begin{split}
& \biggl( \frac{1}{2\sqrt{2}}, 0.833289 - 0.552838 i, -0.724402 - 0.689378 i, 0.951773 - 0.306802 i ,  \\
 &  0.951773 + 0.306802 i, -0.724402 + 0.689378 i, 0.833289 + 0.552838 i \biggr) \nonumber
\end{split}
\end{equation}

\item $n = 11$, $d = \frac{1}{2\sqrt{3}}$\\
\vspace{-0.2cm}
\begin{equation}
\begin{split}
& \biggl( \frac{1}{2\sqrt{3}}, -0.00724338 - 0.999974 i, 0.760473 - 0.649369 i, 0.534533 - 0.845148 i,  \\ 
 &-0.750986 + 0.660318 i, 0.906599 - 0.421993 i, 0.906599 + 0.421993 i, -0.750986 - 0.660318 i,  \\ 
 &0.534533 + 0.845148 i, 0.760473 + 0.649369 i, -0.00724338 + 0.999974 i \biggr) \nonumber
\end{split}
\end{equation}

\item $n = 13$, $d = \frac{5}{2\sqrt{3}}$\\
\vspace{-0.2cm}
\begin{equation}
\begin{split}
 &\biggl( \frac{5}{2\sqrt{3}}, 0.153536 - 0.988143 i, -0.704051 - 0.710149 i, 0.595162 - 0.803606 i, \\ 
 & 0.869485 - 0.493959 i, 0.0560739 + 0.998427 i, 0.184495 - 0.982834 i, 0.184495 +  0.982834 i,  \\ 
 & 0.0560739 - 0.998427 i,0.869485 + 0.493959 i, 0.595162 + 0.803606 i, -0.704051 + 0.710149 i, \\
 & 0.153536 + 0.988143 i \biggr) \nonumber
\end{split}
\end{equation} 

\item $n = 15$, $d = \frac{1}{4}$\\
\vspace{-0.2cm}
\begin{equation}
\begin{split}
 &\biggl( \frac{1}{4}, 0.989074 + 0.147421 i, 0.0432273 - 0.999065 i, -0.309017 - 0.951057 i, \\ 
 & 0.165435 + 0.986221 i, -0.5 - 0.866025 i, 0.809017 - 0.587785 i, 0.552264 - 0.833669 i,  \\ 
 &0.552264 + 0.833669 i, 0.809017 + 0.587785 i, -0.5 +  0.866025 i, 0.165435 - 0.986221 i, \\
 & -0.309017 + 0.951057 i, 0.0432273 + 0.999065 i, 0.989074 - 0.147421 i \biggr) \nonumber
\end{split}
\end{equation}

\item $n = 19$, $d = \frac{1}{2\sqrt{5}}$\\
\vspace{-0.2cm}
\begin{equation}
\begin{split}
 &\biggl( \frac{1}{2\sqrt{5}}, 0.999747 - 0.0225052 i, -0.660552 - 0.75078 i, 0.56565 - 0.824645 i, \\ 
 & -0.693668 - 0.720295 i, 0.527969 - 0.849264 i, 0.952885 + 0.303331 i, -0.0601301 + 0.998191 i,  \\ 
 & 0.802764 - 0.596297 i, -0.422203 - 0.906501 i, -0.422203 + 0.906501 i, 0.802764 + 0.596297 i, \\
 & -0.0601301 - 0.998191 i, 0.952885 - 0.303331 i, 0.527969 + 0.849264 i, -0.693668 + 0.720295 i,  \\ 
 & 0.56565 + 0.824645 i, -0.660552 + 0.75078 i, 0.999747 + 0.0225052 i \biggr) \nonumber
\end{split}
\end{equation}

\item $n = 21$, $d = \frac{11}{4}$\\
\vspace{-0.2cm}
\begin{equation}
\begin{split}
 &\biggl(  \frac{11}{4}, -0.643041 + 0.765832 i, 0.521717 - 0.853118 i, 0.38874 - 0.921348 i, \\ 
 &  0.247078 - 0.968996 i, -0.999681 + 0.0252613 i, 0.0495156 + 0.998773 i, 0.5 - 0.866025 i, \\ 
 & -0.341709 - 0.939806 i, 0.811745 - 0.584012 i, 0.715636 - 0.698474 i, 0.715636 + 0.698474 i, \\
 & 0.811745 + 0.584012 i, -0.341709 + 0.939806 i, 0.5 +  0.866025 i, 0.0495156 - 0.998773 i,  \\ 
 &-0.999681 - 0.0252613 i, 0.247078 + 0.968996 i, 0.38874 + 0.921348 i, 0.521717 + 0.853118 i, \\
 & -0.643041 - 0.765832 i \biggr) \nonumber
\end{split}
\end{equation}

\end{itemize}



\begin{thebibliography}{99}

\bibitem{A05}
D.M. Appleby: SIC-POVMs and the Extended Clifford Group. \textit{J. Math. Phys.} \textbf{46} (2005) 052107.

\bibitem{B89}
J. Backelin: Square multiples $n$ give infinitely many cyclic $n$-roots, Reports, Matematiska Institutionen No 8, Stockholms Universitet (1989).

\bibitem{BBRV02}
S. Bandyopadhyay, P.O. Boykin, V. Roychowdhury, and F. Vatan: A new proof of the existence of mutually unbiased bases. \textit{Algorithmica} \textbf{34} (2002) 512–-528.

\bibitem{blake2014}
S.T. Blake and A.Z. Tirkel: A construction for perfect periodic autocorrelation sequences. In: \textit{Int. Conf. Seq. Their Appl. -SETA 2014} Springer (2014) 104--108.

\bibitem{Br}
R.A. Brualdi: A note on multipliers of difference sets. \textit{J. Res. Natl. Bur. Stand.} \textbf{69B} (1965) 87--89.

\bibitem{chu1972}
D. C. Chu: Polyphase Codes With Good Periodic Correlation Properties. \textit{IEEE Trans. Inf. Theory} \textbf{18} (4) (1972) 531--532.

\bibitem{CK}
R. Craigen, H. Kharaghani: On the nonexistence of Hermitian circulant complex Hadamard matrices. \textit{Australas. J. Combin.} \textbf{7} (1993) 225--227.

\bibitem{Cr}
R. Craigen: Trace, symmetry and orthogonality. \textit{Canad. Math. Bull.} \textbf{37} (1994) 461--467.

\bibitem{F01}
J.C. Faug\`{e}re: Finding all the solutions of Cyclic 9 using Gr\"{o}bner basis techniques, Lecture Notes Ser. Comput. \textbf{9} (2001) 1-–12.

\bibitem{farnertt1990}  
E.C. Farnett et al: Pulse Compression Radar. \textit{Radar Handbook}, 2nd edition, Skolnik, M., Ed., McGraw-Hill (1990).

\bibitem{HS16}
G. Hiranandani, J.M. Schlenker: Small circulant complex Hadamard matrices of Butson type. \textit{Europ. J. Combinatorics} \textbf{51} (2016) 306--314.

\bibitem{heimiler1961}
R.C. Heimiller: Phase Shift Pulse Codes with Good Periodic Correlation Properties. \textit{IRE Transactions on Information Theory} \textbf{7} (4) (1961) 254--257.

\bibitem{I81}
I.D. Ivanovic: Geometrical description of quantum state determination. \textit{J. Physics A} \textbf{14} (12) (1981) 3241–-3245.

\bibitem{itapov2005}  
V.P. Ipatov: \textit{Spread Spectrum and CDMA: Principles and Applications}. John Wiley \&	Sons (2005).

\bibitem{Jo}
E.C.Johnsen: The inverse multiplier for abelian group difference sets. \textit{Canad. J. Math.} \textbf{16} (1964) 787--796.

\bibitem{La}
C.W.H. Lam: Non-skew symmetric orthogonal matrices with constant diagonals. \textit{Discrete Math.} \textbf{43} (1983) 65--78.

\bibitem{LF2004}
Y. Liu and P. Fan: Modified Chu sequences with smaller alphabet size. \textit{Electron. Lett.} \textbf{40} (10) (2004) 598--599.

\bibitem{milewski1983}  
A. Milewski: Periodic Sequences with Optimal Properties for Channel Estimation and Fast Start-Up Equalization. \textit{IBM J. Res. Dev.} \textbf{27} (5) (1983) 426--431.

\bibitem{MKW}
J.H. McKay, S.S.-S. Wang: On a theorem of Brualdi and Newman. \textit{Linear Algebra Appl.} \textbf{92} (1987) 39--43.

\bibitem{mow1993}
W.H. Mow: \textit{A Study of Correlation of Sequences}. PhD, Department of Information Engineering, The Chinese University of Hong Kong (1993).

\bibitem{R63}
H.J. Ryser: \textit{Combinatorial mathematics}. The Carus Mathematical Monographs, No. 14, Published by The Mathematical Association of America; distributed by John Wiley and Sons, Inc., New York (1963).

\bibitem{S06}
A. Scott: Tight informationally complete quantum measurements. \textit{J. Phys. A: Math. Gen.} \textbf{39} (2006) 13507.

\bibitem{Sch1}
B. Schmidt: Cyclotomic integers and finite geometry. \textit{J. Am. Math. Soc.} \textbf{12} (1999) 929--952.

\bibitem{Sch2}
B. Schmidt: Towards Ryser's conjecture. In: C. Casacuberta et al., eds.: \textit{Proc. of 3rd European Congress on Mathematics}, Progress in Mathematics, vol. 201, Birkh\"auser 2001, pp. 533--541.

\bibitem{SL}
J. Seberry, C.W.H. Lam: On orthogonal matrices with constant diagonal. \textit{Linear Algebra Appl.} \textbf{46} (1982) 117--129.

\bibitem{SM}
R.G. Stanton, R.C. Mullin: On the nonexistence of a class of circulant balanced weighing matrices. \textit{SIAM J. Appl. Math.} \textbf{30} (1976) 98--102.

\bibitem{TG}
O. Turek, D. Goyeneche: A generalization of circulant Hadamard and conference matrices. \textit{Linear Algebra Appl.} \textbf{569} (2019) 241--265.

\bibitem{Tu}
R. Turyn: Character sums and difference sets. \textit{Pacific J. Math.} \textbf{15} (1965) 319--346.

\bibitem{WF89}
W.K. Wootters and B.D. Fields: Optimal state–determination by mutually unbiased measurements. \textit{Annals of Physics} \textbf{191} (2) (1989) 363–-381.

\bibitem{xu2011}
L. Xu: Phase coded waveform design for Sonar Sensor Network. \textit{Conference on Communications and Networking in China (CHINACOM), 2011 6th International ICST}, pp. 251-256, August 2011.


\end{thebibliography}
\end{document}